\documentclass[12pt]{amsart}
\usepackage{amsmath, amssymb, amsthm}
\usepackage{paralist, xcolor, hyperref}

\usepackage[margin=1in,letterpaper,portrait]{geometry}

\newcommand{\RR}{\mathbb{R}}

\newcommand{\ZZ}{\mathbb{Z}}

\newcommand{\affS}{\widetilde{S}}
\newcommand{\wt}{\widetilde}
\newcommand{\ol}{\overline}
\newcommand{\tvd}{\operatorname{dis}}
\renewcommand{\neg}{\operatorname{neg}}
\newcommand{\depth}{\operatorname{dp}}
\newcommand{\reflen}{\operatorname{reflen}}
\newcommand{\bl}{\operatorname{bl}}

\newtheorem{conjecture}{Conjecture}[section]
\newtheorem{theorem}[conjecture]{Theorem}
\newtheorem{corollary}[conjecture]{Corollary}
\newtheorem{lemma}[conjecture]{Lemma}
\newtheorem{proposition}[conjecture]{Proposition}

\theoremstyle{definition}
\newtheorem{question}[conjecture]{Question}
\newtheorem{example}[conjecture]{Example}
\newtheorem{remark}[conjecture]{Remark}
\newtheorem{definition}[conjecture]{Definition}
\newtheorem*{question*}{Question}

\newcommand{\bm}[1]{\mbox{\boldmath $#1$}}

\usepackage{stmaryrd}

\title{Bargain hunting in a Coxeter group}

\author[Joel Brewster Lewis]{Joel Brewster Lewis$^*$}
\address{Department of Mathematics, George Washington University, Washington, DC, USA}
\email{jblewis@gwu.edu}
\thanks{$^*$Research partially supported by a grant from the Simons Foundation (634530).}

\author[Bridget Eileen Tenner]{Bridget Eileen Tenner$^\dagger$}
\address{Department of Mathematical Sciences, DePaul University, Chicago, IL, USA}
\email{bridget@math.depaul.edu}
\thanks{$^\dagger$Research partially supported by NSF Grant DMS-2054436.}

\begin{document}

\maketitle

\begin{abstract}
Petersen and Tenner defined the depth statistic for Coxeter group elements which, in the symmetric group, can be described in terms of a cost function on transpositions.  We generalize that cost function to the other classical (finite and affine) Weyl groups, letting the cost of an individual reflection $t$ be the distance between the integers transposed by $t$ in the combinatorial representation of the group (\`a la Eriksson and Eriksson).  Arbitrary group elements then have a well-defined cost, obtained by minimizing the sum of the transposition costs among all factorizations of the element.  We show that the cost of arbitrary elements can be computed directly from the elements themselves using a simple, intrinsic formula. 
\end{abstract}

\section{Introduction}

Given any group $W$, any generating set $T$ for $W$, and any \emph{cost function}
\[
    \$ : T \rightarrow \RR_{>0}
\]
on the generators, one can extend the cost function to all of $W$ by minimizing over all decompositions of $w$ into products of elements of $T$:
\begin{equation*}
    \$(w) = \min_{\substack{t_1,\ldots, t_k \in T: \\ t_1\cdots t_k = w}}  \big\{\$(t_1) + \cdots + \$(t_k)\big\}.
\end{equation*}
One family of common examples occurs in the case that $\$(t) = 1$ for all $t \in T$: then $\$(w)$ is the minimum length of an expression for $w$ as a product of generators.  For example, when $W = S_n$ is the symmetric group and $T = S := \{(1~2), (2~3), \ldots, (n - 1~n)\}$ is the set of simple transpositions, we have that $\$(w)$ is the length of the shortest possible expressions for $w$ as a product of simple transpositions, which is also known to be the \emph{inversion number} $\ell_S(w)$ of $w$.  If, instead, $T = \{ (1~2), (2~3), (1~3), \ldots\}$ consists of \emph{all} transpositions in $S_n$ and $\$(t) = 1$ for all $t \in T$, then $\$(w)$ is the \emph{reflection length} $\reflen(w)$ of $w$, which is also known as the \emph{absolute length}. This is known to equal $n - c(w)$, where $c(w)$ is the number of cycles of the permutation $w$.

In \cite{PT}, the authors considered the situation in which $S_n$ is generated by the set $T = \{(i~j)\}$ of all transpositions, and the cost function on transpositions is given by
\begin{equation}
\label{eq:difference cost}
    \$((i~j)) = |j-i|,
\end{equation}
i.e., the cost of a transposition is equal to the distance between the points it transposes in the one-line notation of a permutation.  They showed that for this function, the cost of a permutation is half of its \emph{total displacement}:
\begin{equation}
\label{eq:PT answer}
    \$(w) = \frac{1}{2} \sum_{i=1}^n |w(i)-i|.
\end{equation}

One theme of all of these examples is that a certain ``extrinsic,'' ``extremal'' quantity (the minimum of a cost function over the family of factorizations of a given element) can also be given by a simple ``intrinsic'' formula that can be computed directly from the element.  The main thesis of this note is that formulas of this kind are beautiful and interesting, and that it would be nice to have more of them.

The symmetric groups are Coxeter group of finite type $A$, and the three preceding examples can all be rephrased at the level of generality of arbitrary Coxeter groups.  In particular, in \cite{PT}, it was shown that the cost $\$((i~j)) = |j - i|$ can be described in terms of the associated root system: the transpositions in $S_n$ are the reflections, and the cost $|j - i|$ is the \emph{depth} of the positive root associated to the reflection $(i~j)$ in the root system.  This gives two perspectives to the work of \cite{PT}: the combinatorial view focuses on the definition of the cost function given in Equation~\eqref{eq:difference cost}, while the algebraic view focuses on the definition of the cost function in terms of root depths.

In \cite{BBNW}, the authors studied the depth-defined cost function for the Coxeter groups of types $B$ and $D$ (the signed and even-signed permutations). They were able to give formulas for the cost of arbitrary elements with this depth-defined cost function, but their results were, in a sense, less tidy than the results cited above. 
This raises the following question.

\begin{question*}
Are there natural choices of cost functions on the reflections in types $B$ and $D$, or on other related groups like the affine symmetric group, whose extension to the group is given by a simple, attractive formula?
\end{question*}

Our main results are to answer this question in the affirmative, generalizing the combinatorial perspective on the cost function of Equation~\eqref{eq:difference cost}.  We work in the level of generality of the infinite families of Weyl and affine Weyl groups---finite types $A$, $B$, $D$ and affine types $\wt{A}$, $\wt{B}$, $\wt{C}$, $\wt{D}$---which Eriksson and Eriksson called \emph{George groups} \cite{EE}.  These groups share a common combinatorial description as permutation groups acting on $\ZZ$ that commute with certain symmetries of the number line.  

The details of these groups and of our cost function \$, which is motivated directly by the cost given in Equation~\eqref{eq:difference cost}, are presented in Section~\ref{sec:background}. The main results of our work are a pair of theorems---Theorem~\ref{thm:unbranched} and Theorem~\ref{thm:finite type}---giving the cost of arbitrary elements in the unbranched and finite George groups, respectively. Those results show that the cost of an arbitrary element can be computed directly from that element using a simple, intrinsic formula. Finally, in Section~\ref{sec:further remarks}, we make some further remarks and state some open problems related to our work.

\section{Background}
\label{sec:background}
\subsection{Who are the groups?}

The main objects of this paper are the \emph{classical (finite and affine) Coxeter groups}, or, in the language of \cite{EE}, the \emph{George groups}.  Each of these groups consists of bijections from (a subset of) $\ZZ$ to itself that satisfy certain symmetry conditions.  We describe them now, following \cite{BB, EE}.

\begin{itemize}
\item The \emph{symmetric group} $S_n$ is a Coxeter group of (finite) type $A$.  It consists of all bijections from $[n] = \{1, \ldots, n\}$ to itself.  

\item The \emph{hyperoctahedral group} $S^B_n$ is a Coxeter group of (finite) type $B$ (and also finite type $C$). Its elements are the \emph{signed permutations}, which are the bijections from $\pm[n] := \{\pm 1, \ldots, \pm n\}$ to itself satisfying the symmetry condition $w(i) = -w(-i)$. 

\item The group $S^D_n$ is a Coxeter group of (finite) type $D$. It is the subgroup of $S^B_n$ consisting of the \emph{even-signed permutations}, those for which $\#\{i \in [n] : w(i) < 0\}$ is even.

\item The \emph{affine symmetric group} $\wt{S}_n$ is a Coxeter group of affine type $\wt{A}$.  Its elements are the \emph{affine permutations}. These are bijections $w: \ZZ \to \ZZ$ such that
\begin{itemize}[$\star$]
    \item $w(i + n) = w(i) + n$ for all $i \in \ZZ$, and
    \item $w(1) + \cdots + w(n) = \binom{n+1}{2}$.
\end{itemize}

\item The Coxeter group $\wt{S}^C_n$ of affine type $\wt{C}$ consists of the \emph{affine signed permutations}. In the language of \cite{EE} (which differs from that in \cite{BB}---see Section~\ref{sec:BB vs EE}), these are the bijections $w: \ZZ \to \ZZ$ such that
\begin{itemize}[$\star$]
    \item $w(-i) = -w(i)$ for all $i \in \ZZ$, and
    \item $w(i + (2n + 2)) = w(i) + (2n + 2)$ for all $i \in \ZZ$.
\end{itemize}

\item The Coxeter groups $\wt{S}^B_n$ and $\wt{S}^D_n$ of affine types $\wt{B}$ and $\wt{D}$ are subgroups of $\wt{S}^C_n$ satisfying some evenness conditions analogous to the condition in finite type $D$.
\end{itemize}

We will refer to these groups collectively as the \emph{George groups of window size $n$}. Due to the symmetry conditions in these groups, each element $w$ is uniquely described by the \emph{window} of data $[w(1), w(2), \ldots, w(n)]$. For example, for $w = [-5, 6, 7] \in \wt{S}^C_3$ one has $w(4) = 4$, $w(5) = w(-3) + 8 = -w(3)+8 = 1$, $w(6) = w(-2) + 8 = -w(2) + 8 = 2$, $w(7) = w(-1) + 8 = -w(1) + 8 = 13$, $w(8) = 8$, and so on.

Our work focuses on these five groups $S_n$, $S^B_n$, $S^D_n$, $\wt{S}_n$, and $\wt{S}^C_n$, which we will think of as the union of two (overlapping) classes:
\begin{itemize}
\item the \emph{unbranched} George groups $S_n$, $S^B_n$, $\wt{S}_n$, and $\wt{S}^C_n$, and
\item the \emph{finite} George groups $S_n$, $S^B_n$, and $S^D_n$.
\end{itemize}
The ``unbranched'' terminology refers to the fact that the Dynkin diagrams for these groups have no vertices of degree greater than $2$ (see \cite[Table 1]{EE} or \cite[\S A1]{BB}). We will say more about the remaining two groups $\wt{S}^B_n$ and $\wt{S}^D_n$ in Section~\ref{sec:affine B and D}. 

It is common to ease notation by writing $\overline{i} := -i$.  The symmetry conditions satisfied by a George group $W$ divide its domain into \emph{symmetry classes}.  For $S_n$, with no symmetry conditions, these are simply the singleton classes $\{1\}$, \ldots, $\{n\}$.  The single symmetry condition satisfied by $S^B_n$ and $S^D_n$ yields $n$ symmetry classes $\{1, \ol{1}\}$, \ldots, $\{n, \ol{n}\}$.  For the affine symmetric group $\wt{S}_n$, the symmetry classes are precisely the congruence classes of integers modulo $n$.  Affine signed permutations satisfy two symmetry conditions: elements of $\wt{S}^C_n$ are the bijections on $\ZZ$ that commute with reflection across $0$ and translation by $2(n + 1)$.  It follows that $w((2n + 2) - i) = (2n + 2) - w(i)$ for $w$ in this group, so the affine signed permutations commute with reflection across $n + 1$.  By translation, affine signed permutations commute with reflection across $k(n + 1)$ for all $k$, and so $k(n + 1)$ is a fixed point of every affine signed permutation for all $k$.  Thus, the symmetry classes for the remaining groups $\wt{S}^C_n$, $\wt{S}^B_n$, and $\wt{S}^D_n$ consist of the nontrivial classes $\{ m \colon m \equiv i \text{ or } \ol{i} \pmod{2n + 2}\}$ for $i = 1, \ldots, n$, as well as the trivial classes $\{k(n + 1)\}$ for $k \in \ZZ$.

\subsection{What are the statistics?}
\label{sec:statistics}

We will study several permutation statistics in this article, and we highlight the key definitions here.  Throughout this paper, $n$ will be an arbitrary but fixed positive integer.

We begin with a definition from \cite{EE}, giving a generic framework of transpositions that works for all of the groups we study, rather than naming the particular transpositions that apply in each case.

\begin{definition}
Given a George group $W$, a pair $\{i,j\}$ of different positions in $\ZZ$ is \emph{transposable} if there exists at least one element $w$ in $W$ for which $w(i) = j$ and $w(j) = i$. If $\{i,j\}$ is a transposable pair, then the \emph{transposition} $\langle (i~j ) \rangle$ is the extension of the transposition $(i~j)$ under the symmetry conditions of the group; this is always an element of $W$ of multiplicative order $2$.
\end{definition}

The transpositions in a George group $W$ are precisely the \emph{reflections} (when $W$ is viewed as a Coxeter group).  In what follows, it will be useful to divide the transpositions into two flavors.  First, all George groups contain transpositions of the form $\langle (i~j)\rangle$ where $i$ and $j$ belong to different symmetry classes.   Second, some groups contain transpositions that switch two integers in the \emph{same} symmetry class: in particular, the groups $S^B_n$ and $\wt{S}^C_n$ contain transpositions of the form $\langle (i~\ol{i})\rangle$, and the groups $\wt{S}^B_n$ and $\wt{S}^C_n$ contain transpositions of the form $\langle (i~2n+2 - i) \rangle$. 

In each of the George groups of window size $n$, there is a natural generating set of \emph{simple reflections} that makes it a Coxeter group. 
For $i \in [n]$, define $s_i := \langle (i~i+1) \rangle$. Additionally, we let $s_0 := \langle (1~\ol{1}) \rangle$, $s_1' := \langle (1~\ol{2}) \rangle$, and $s_n' := \langle (n~n+2) \rangle$.
Then
\begin{itemize}
\item $S^A_n := S_n$ is generated by $\{s_1, \ldots, s_{n-1}\}$,
\item $S^B_n$ is generated by $\{s_0, s_1, \ldots, s_{n-1}\}$,
\item $S^D_n$ is generated by $\{s_1', s_1, \ldots, s_{n-1}\}$,
\item $\wt{S}_n$ is generated by $\{s_1, \ldots, s_{n-1}, s_n\}$, and
\item $\wt{S}^C_n$ is generated by $\{s_0, s_1, \ldots, s_{n-1}, s_n'\}$.
\end{itemize}

In the symmetric group, one often considers \emph{inversions}, which are pairs of elements of $[n]$ that are out of order.  The following definition extends this notion to all George groups.

\begin{definition}
Let $w$ be an element in a George group.  A (right) \emph{class inversion} in $w$ is a transposition $\langle (i~j) \rangle$ such that $i < j$ and $w(i) > w(j)$. 
\end{definition}

The \emph{length} $\ell(w)$ of an element of a Coxeter group is the smallest number of simple reflections needed to multiply to give $w$.  In the symmetric group, this is equal to the number of inversions of $w$.  By \cite[Thm.~15]{EE}, the previous definition extends this to all George groups: for any element $w$ in any George group, the length $\ell(w)$ in that group is equal to the number of its class inversions.

Because an element of a George group of window size $n$ is entirely determined by the data in its window $[w(1), \ldots, w(n)]$, we can define the following statistic on all such elements.

\begin{definition}
\label{def:tvd}
If $w$ belongs to any of the George groups of window size $n$, then the \emph{total displacement} \cite{Knuth} or \emph{Spearman's disarray} \cite{DG} of $w$ is
\[
\tvd(w) := \sum_{i = 1}^n |w(i) - i|.
\]
\end{definition}
\begin{remark}
\label{changing the index set}
When $W$ respects a symmetry group on a subset of $\ZZ$ (made up of some set of reflections and translations), and if $i$ and $j$ are in the same symmetry class, then $|w(i) - i| = |w(j) - j|$.  Consequently, in Definition~\ref{def:tvd}, the summation index set $[n]$ may be replaced by any set that contains exactly one element in the same symmetry class as $i$ for each $i$ in $[n]$.
\end{remark}

As mentioned in the introduction, the statistic studied in this paper is the cost function $\$(w)$, defined as follows:  if $t$ is a transposition in a George group, then 
$$\$(t) = \frac{\tvd(w)}{2},$$
while for an arbitrary element $w$ we define
\begin{equation}
\label{eq:optimization}
    \$(w) = \min_{\substack{t_1\cdots t_k = w \\ t_1,\ldots, t_k \in T}}  \big\{\$(t_1) + \cdots + \$(t_k)\big\}.
\end{equation}
Concretely, the first condition means that $\$(\langle (i~j)\rangle) = |i - j|$ if $i$ and $j$ belong to different symmetry classes, while $\$(\langle (i~j)\rangle) = \frac{1}{2}|i - j|$ if $i$ and $j$ belong to the same symmetry class.  For example, in the hyperoctahedral group $S^B_n$, for $i, j$ distinct elements of $[n]$ we have
\begin{itemize}
    \item $\$( (i~j)(\ol{i}~\ol{j})) = |i - j|$, 
    \item $\$( (i~\ol{i})) = i$, and
    \item $\$( (i~\ol{j})(\ol{i}~j)) = i + j$. 
\end{itemize}

In \cite{PT} and subsequently \cite{BBNW}, the authors consider a statistic on an arbitrary Coxeter group $W$ that they call \emph{depth}.  For a reflection $t$, the depth is defined to be\footnote{This is the same as the depth in the root system of $W$ of the positive root orthogonal to the reflecting hyperplane of $t$.}
\[
\depth(t) = \frac{1 + \ell(t)}{2},
\]
and depth is extended to all of $W$ by the same minimization as in Equation~\eqref{eq:optimization}.  In the case of the symmetric group, one has that $\ell( (i~j)) = 2|j - i| - 1$.  Thus $\depth( (i~j )) = \frac{1}{2} \tvd((i~j))$ for every transposition $(i~j)$ in $S_n$, and consequently $\depth(w) = \$(w)$ for every permutation $w \in S_n$; the main result of \cite{PT} establishes that $\depth(w) = \$(w) = \frac{1}{2} \tvd(w)$ for $w \in S_n$.  In types $B$ and $D$, we no longer have that $\depth(t)$ and $\$(t)$ are equal for all transpositions $t$; in particular, for a transposition $ \langle (i \ \ol{j}) \rangle = (i \ \ol{j})(\ol{i} \ j)$ with $i, j > 0$, we have $\$(\langle (i \ \ol{j}) \rangle) = i + j$ while $\depth^B(\langle (i \ \ol{j}) \rangle) = i + j - 1$ and $\depth^D(\langle (i \ \ol{j}) \rangle) = i + j - 2$. Thus, the problems of computing, for arbitrary $w$ in $S^B_n$ or $S^D_n$, the values of $\depth(w)$ and $\$(w)$ are different.  The main result of \cite{BBNW} is to give formulas for $\depth(w)$ for $w$ in $S^B_n$ or $S^D_n$, in terms of another statistic they call the \emph{blocks} of the signed permutation, which we will discuss in Section~\ref{sec:finite}.

\section{Main theorems}
\label{sec:main theorems}
The main results of this paper are about the cost of elements in the George groups. We will state and prove this cost first for the unbranched George groups, because arguments in those cases are susceptible to the same proof techniques. We will then state and prove analogous results for the finite George groups, again taking advantage of commonalities in the approaches for those cases. These two classes of groups overlap in $S_n$ and $S^B_n$, and, of course, the results are the same for those groups whether they are considered unbranched or finite.

\subsection{First main theorem: unbranched George groups}
\label{sec:unbranched}

In this section, we prove our first main theorem, which applies to the unbranched George groups $S_n$, $S^B_n$, $\affS_n$, and $\affS^C_n$.

\begin{theorem}
\label{thm:unbranched}
    For any unbranched George group $W$ and any element $w$ in $W$, we have
    \[
    \$(w) = \frac{\tvd(w)}{2}.
    \]
\end{theorem}

In the case of $S_n$, Theorem~\ref{thm:unbranched} recovers \cite[Theorem~1.1]{PT}. Our proof here will be more akin to the proof in \cite{BBNW} than the one in~\cite{PT}. In fact, the proof of Theorem~\ref{thm:unbranched} will follow the same general outline for each of the four unbranched George groups: first we will show that the cost $\$(w)$ is at least as large as $\tvd(w)/2$; then we will show that a particular transposition can always be found in the unbranched George groups; finally, we will show that when that particular kind of transposition exists, equality can be achieved. The result will follow by induction.

\begin{proposition}
    \label{prop:subadditive}
If $u$ and $w$ belong to any George group $W$, then 
\[
\tvd(uw) \leq \tvd(u) + \tvd(w).
\]    
\end{proposition}

\begin{proof}
Fix $u$ and $w$ in a George group $W$ of window size $n$. Then by definition and the triangle inequality, we have
\begin{align*}
\tvd(uw) & = \sum_{i = 1}^n |u(w(i)) - i| \\
 & = \sum_{i = 1}^n |u(w(i)) - w(i) + w(i) - i| \\
 & \leq \sum_{i = 1}^n |u(w(i)) - w(i)| + \sum_{i = 1}^n|w(i) - i|.
\end{align*}
Since $W$ is a George group, the set $\{w(1), \ldots, w(n)\}$ contains exactly one element in the symmetry class of $i$ for each $i$ in $[n]$.  Thus, by Remark~\ref{changing the index set}, $ \sum_{i = 1}^n |u(w(i)) - w(i)| = \tvd(u)$, and so $\tvd(uw) \leq \tvd(u) + \tvd(w)$, as claimed.
\end{proof}

We can now prove one direction of Theorem~\ref{thm:unbranched}.

\begin{corollary}
\label{cor:lower bound}
If $w$ belongs to any George group, then
\[
\$(w) \geq \frac{\tvd(w)}{2}. 
\]
\end{corollary}

\begin{proof}
Given $w$, choose a minimum-cost transposition factorization $w = t_1 \cdots t_k$ of $w$; that is, $\$(w) = \$(t_1) + \cdots + \$(t_k)$.  By definition, $\$(t) = \tvd(t)/2$ for every transposition $t$.  Thus, by Proposition~\ref{prop:subadditive}, we have 
\[
\$(w) = \frac{\tvd(t_1) + \cdots + \tvd(t_k)}{2} \geq \frac{\tvd(t_1 \cdots t_k)}{2} = \frac{\tvd(w)}{2},
\]
as claimed.
\end{proof}

While Proposition~\ref{prop:subadditive} and Corollary~\ref{cor:lower bound} apply to elements of all George groups, the next result requires that the group under consideration be unbranched. Moreover, this is the step in our proof of Theorem~\ref{thm:unbranched} that involves a case analysis, by group. It is notable that the affine cases are somewhat delicate, and use explicit descriptions of which pairs are transposable in their groups. 

\begin{lemma}\label{lem:unbranched can find a good transposable pair}
Let $W$ be an unbranched George group ($S_n$, $S^B_n$, $\wt{S}_n$, or $\wt{S}^C_n$).
If $w$ is a non-identity element of $W$, then there exists a transposable pair $\{x,y\}$ for $W$ such that
$$
w(x) \ge y > x \ge w(y).
$$
\end{lemma}

\begin{proof}
We give separate proofs for the different groups. All cases rely on the fact that $w$ is not the identity, and so it has some non-fixed values. 

\vspace{.1in}

\noindent \framebox{Case 1: $W = S_n$} 
Choose $y$ so that $w(y)$ is the smallest non-fixed value. The minimality means that, in particular, $w(y) < y$. Thus there is a value (namely, $w(y)$) less than $y$ and in a position weakly to the right of $y$ in the window of $w$ (in fact, in position $y$). Therefore there must be a value weakly larger than $y$ and appearing strictly to the left of position $y$. That is, there exists $x < y$ with $w(x) \ge y$. Moreover, the definition of $y$ means that $x \in [w(y), y-1]$, so, in fact, we have
$$w(x) \ge y > x \ge w(y).$$

\vspace{.1in}

\noindent \framebox{Case 2: $W = S^B_n$} 
The same proof works, with two adjustments: all references to positions in the window should refer to the ``doubled window'' $[w(-n), \ldots, w(-1), w(1), \ldots, w(n)]$, and $w(y)$ should now be chosen as the smallest non-fixed value in the doubled window.

\vspace{.1in}

\noindent \framebox{Case 3: $W = \wt{S}_n$} 
Divide the non-fixed values of $w$ into two sets: the exceedances $E := \{i \in \ZZ : w(i) > i\}$ and the anti-exceedances $A := \{j \in \ZZ : w(j) < j\}$.  Because $w$ is not the identity, $A \cup E$ is nonempty. Suppose, without loss of generality, that $A$ is nonempty. (The argument in the case that $E$ is nonempty is entirely analogous.)  By the periodicity property of $w$, we have $i \in A$ if and only if $i + kn \in A$, so $A$ has at least one element in $[n]$.  But then because $\sum_{i=1}^n w(i) = \sum_{i=1}^n i$, we must have that $E$ is nonempty as well.  Combining this with the periodicity, we have that for each element $j$ in $A$, there is some element $i$ in $E$ with $i < j$.  Consequently, there exists a position $y \in A$ for which the largest element of $A \cup E$ less than $y$ is an element of $E$; call the position of this exceedance $x$.

By the choice of $x$ and $y$, we have $y > x$, $y > w(y)$, and $w(x) > x$.  Moreover, again by the choice of $x$ and $y$ we have that $w(z) = z$ for all $z$ in $\{x + 1, \ldots, y - 1\}$.  Since $w$ is a bijection, $w(x)$ cannot be equal to any of $w(x + 1) = x+ 1, \ldots, w(y - 1) = y - 1$, and therefore $w(x) \geq y$; and similarly $w(y) \leq x$.  Thus 
\[
w(x) \geq y > x \geq w(y),
\]
as claimed.  
Finally, since $y \in A$, every number of the form $y + kn$ also belongs to $A$.  Since $x$ belongs to $E$ (which is disjoint from $A$), $x$ does not differ from $y$ by a multiple of $n$, and therefore $\{x, y\}$ is a transposable pair.

\vspace{.1in}

\noindent \framebox{Case 4: $W = \wt{S}^C_n$} 
Define the sets $E$ and $A$ as in Case 3. Again, because $w$ is not the identity, we have $A \cup E \neq \emptyset$.
Suppose that $i \in E$, so that $w(i) > i$.  Then by the first symmetry property of $w$, we have that $w(-i) = -w(i) < -i$, so $-i \in A$.  Moreover, since $w(i + k(2n + 2)) = w(i) + k(2n + 2)$ for all $k$, we have $i + k(2n + 2) \in E$ and $-i + k(2n + 2) \in A$ for all $k$.  Similarly, if $j \in A$ then $j + k(2n + 2) \in A$ and $-j + k(2n + 2) \in E$ for all $k$.  
It follows that when $E \cup A$ is nonempty, both $E$ and $A$ include arbitrarily large and arbitrarily small elements.  Therefore, for each element $j$ in $A$, there is some element $i$ in $E$ with $i < j$.  Consequently, there exists a position $y \in A$ for which the largest element of $A \cup E$ less than $y$ is an element of $E$; call this position $x$.  

By the choice of $x$ and $y$, we have $y > x$, $y > w(y)$, and $w(x) > x$.  Moreover, again by the choice of $x$ and $y,$ we have that $w(z) = z$ for all $z$ in $\{x + 1, \ldots, y - 1\}$.  Since $w$ is a bijection, $w(x)$ cannot be equal to any of $w(x + 1) = x+ 1, \ldots, w(y - 1) = y - 1$, and therefore $w(x) \geq y$; and similarly $w(y) \leq x$.  Thus 
\[
w(x) \geq y > x \geq w(y),
\]
as claimed.  The sets $E$ and $A$ are disjoint, with $x \in E$ and $y \in A$. Thus, $x$ does not differ from $y$ by a multiple of $2n + 2$, and therefore $\{x, y\}$ is a transposable pair.
\end{proof}

The last step of our argument is to show that in the presence of a transposable pair such as that described in the statement of Lemma~\ref{lem:unbranched can find a good transposable pair}, we can peel off a transposition from $w$ in an advantageous manner.

\begin{lemma}\label{lem:good transposition can be peeled off}
Suppose that $w$ is an element of a George group $W$ and that $\{x, y\}$ is a transposable pair for $W$ such that $w(x) \geq y > x \geq w(y)$.  Then
\[
\tvd(w \cdot \langle (x \ y)\rangle) = \tvd(w) - \tvd(\langle (x \ y)\rangle).
\]
\end{lemma}
\begin{proof} 
Let $w, x, y$ be as in the statement. Suppose first that $x$ and $y$ do not belong to the same symmetry class.  In this case,
it is possible to choose a set $I$ of $n$ integers that contains both $x$ and $y$ and that contains exactly one element from the symmetry class of $i$ for each $i \in [n]$ (as in Remark~\ref{changing the index set}).  With these choices, we have
\begin{align*}
\tvd(w) & = \sum_{i \in I}|w(i) - i| \\
& = |w(x) - x| + |w(y) - y| + \sum_{i \in I \smallsetminus\{x, y\}}|w(i) - i| \\
& = w(x) - x + y - w(y) + \sum_{i \in I \smallsetminus\{x, y\}}|w(i) - i|.
\end{align*}
Furthermore, because $\langle(x \ y )\rangle$ is a transposition, $\langle(x \ y )\rangle (i) = i$ for $i \in I \smallsetminus\{x, y\}$, and therefore
\begin{align*}
\tvd(w\cdot \langle(x \ y )\rangle) & = \sum_{i \in I}|w(\langle(x \ y )\rangle(i)) - i| \\
& = |w(y) - x| + |w(x) - y| + \sum_{i \in I \smallsetminus\{x, y\}}|w(i) - i| \\
& = x - w(y) + w(x) - y+ \sum_{i \in I \smallsetminus\{x, y\}}|w(i) - i| \\
& = \tvd(w) - 2(y - x).
\end{align*}
Since $x$ and $y$ belong to different symmetry classes, $2(y - x) = \tvd(\langle(x \ y )\rangle)$, which completes the proof in this case.

On the other hand, if $x$ and $y$ belong to the same symmetry class, choose an index set $I$ that contains $y$ and one element from each other symmetry class.  Because $x$ and $y$ are in the same symmetry class, we have $|w(x) - x| = |w(y) - y|$. Thus
\begin{align*}
\tvd(w) & = |w(y) - y| + \sum_{i \in I \smallsetminus\{y\}}|w(i) - i| \\
& = w(x) - x + \sum_{i \in I \smallsetminus\{y\}}|w(i) - i|. 
\end{align*}
Furthermore,
\begin{align*}
\tvd(w\cdot \langle(x \ y )\rangle) 
& = |w(\langle(x \ y )\rangle(y)) - y| + \sum_{i \in I \smallsetminus\{y\}}|w(\langle(x \ y )\rangle(i)) - i| \\
& = w(x) - y + \sum_{i \in I \smallsetminus\{y\}}|w(i) - i| \\
& = \tvd(w) - (y - x).
\end{align*}
Since $x$ and $y$ belong to the same symmetry class, $y - x = \tvd(\langle(x \ y )\rangle)$, as needed.
\end{proof}

We are now prepared to give an inductive proof of Theorem~\ref{thm:unbranched}.

\begin{proof}[Proof of Theorem~\ref{thm:unbranched}]
Thanks to Corollary~\ref{cor:lower bound}, it remains to prove that $\tvd(w)/2$ is at least $\$(w)$ for all $w$. We will prove this by inducting on the length of $w$. Both the cost and the displacement of the identity are $0$, establishing the base case. Now suppose that $\ell(w) > 0$, and assume that the result holds for all elements of length less than $\ell(w)$.

Because $w$ is an element of an unbranched George group $W$, Lemma~\ref{lem:unbranched can find a good transposable pair} means that there is a transposable pair $\{x,y\}$ for $W$ such that
$$w(x) \ge y > x \ge w(y),$$
and Lemma~\ref{lem:good transposition can be peeled off} implies that
$$\tvd(w \cdot \langle(x \ y )\rangle) = \tvd(w) - \tvd(\langle(x \ y )\rangle).$$
Set $v := w \cdot \langle(x \ y )\rangle$.  Since $\langle(x \ y )\rangle$ is a (right) inversion of $w$, we have $\ell(v) < \ell(w)$. Thus the inductive hypothesis applies to $v$, and $\$(v) = \tvd(v)/2$. Since $\langle(x \ y )\rangle$ is a transposition, $\$(\langle(x \ y )\rangle) = \tvd(\langle(x \ y )\rangle)/2$ by definition, and therefore
\begin{align*}
\tvd(w) &= \tvd(v) + \tvd(\langle(x \ y )\rangle)\\
&= 2\cdot \$(v) + 2 \cdot \$(\langle(x \ y )\rangle).
\end{align*}
For any minimal-cost transposition factorization $v = t_1 \cdots t_k$, we have $w = t_1 \cdots t_k \cdot \langle(x \ y )\rangle$, and so
\begin{align*}
\$(w) &\le \$(t_1) + \cdots + \$(t_k) + \$(\langle(x \ y )\rangle)\\
&= \$(v) + \$(\langle(x \ y )\rangle). 
\end{align*}
Combining these results yields
$$2 \cdot \$(w) \leq \tvd(w),$$
completing the proof.
\end{proof}

\subsection{Second main theorem: finite type}
\label{sec:finite}

In this section, we prove our second main theorem, for George groups of finite type ($S^A_n := S_n$, $S^B_n$, and $S^D_n$).  The statement of the theorem involves a statistic introduced in \cite{BBNW}, which we recall now.

Every signed permutation can be expressed uniquely as a direct sum of indecomposable signed permutations (as in \cite[\S2]{BBNW}). For example,
$$[\ol{3},\ol{1},2,\ol{4},7,6,8,\ol{5}] = [\ol{3},\ol{1},2] \oplus [\ol{1}] \oplus [3,2,4,\ol{1}].$$

\begin{definition}
Let $w$ be a signed permutation, with
$$w = w^1 \oplus \cdots \oplus w^k,$$
where each $w^i$ is an indecomposable signed permutation. These $w^1, \ldots, w^k$ are the \emph{type $B$ blocks} of $w$, and
$$\bl^B(w) := k.$$
Of course, even-signed permutations can also be written as direct sums. If we require that the summands themselves be even-signed permutations, then those summands are the \emph{type $D$ blocks} of $w$, and $\bl^D(w)$ is the number of type $D$ blocks required.
\end{definition}

For $w = [\ol{3},\ol{1},2,\ol{4},7,6,8,\ol{5}] \in S_8^D \subset S_8^B$, the type $B$ blocks and the type $D$ blocks of $w$ are given by the following decompositions, respectively:
\begin{align*}
w &= [\ol{3},\ol{1},2] \oplus [\ol{1}] \oplus [3,2,4,\ol{1}]\\
&= [\ol{3},\ol{1},2] \oplus [\ol{1},4,3,5,\ol{2}].
\end{align*}
From this we see that $\bl^B(w) = 3$ while $\bl^D(w) = 2$.

For a (usual, unsigned) permutation $w \in S_n$, the decomposition of $w$ as a direct sum of indecomposables is the same as the decomposition if we think of $w$ as belonging to $S^B_n$ or $S^D_n$.  We say that these indecomposable summands are the \emph{type $A$ blocks} of $w$, and define $\bl^A(w)$ to be the number of such blocks.

\begin{theorem}
\label{thm:finite type}
For every (signed) permutation $w$ in the finite George group $S^X_n$, we have
\[ 
\$(w) = \frac{\tvd(w)}{2} + \bl^B(w) - \bl^X(w).
\]
\end{theorem}
\begin{proof}
For $w$ in $S_n$ or $S^B_n$ (so $X = A$ or $B$), we have by Theorem~\ref{thm:unbranched} that $\$(w) = \tvd(w)/2$, and we have by definition of blocks that the type-$X$ blocks and type-$B$ blocks of $w$ are the same, so the result holds in these cases. 

Let $w$ be a permutation in $S^D_n$, and let $t_1 \cdots t_k$ be a \$-minimizing factorization of $w$ into transpositions.
Recall from Section~\ref{sec:background} that if $i, j > 0$ then $\$((i \ j)(\ol{i}\ \ol{j})) = \depth^D((i \ j)(\ol{i}\ \ol{j}))$ and $\$((i \ \ol{j})(\ol{i}\ j)) = \depth^D((i \ \ol{j})(\ol{i}\ j)) + 2$.  We refer to the transpositions in the second case as \emph{signed}.  
Then, using the Iverson bracket, 
\begin{align*}
\$(w) = \sum_i \$(t_i) 
& = \sum\limits_i \left(\depth^D(t_i) + 2\llbracket t_i \text{ signed}\rrbracket\right)\\
&= \left(\sum\limits_i \depth^D(t_i)\right) + 2\#\{t_i : t_i \text{ is signed}\} \\
&\ge \depth^D(w) + \neg(w).
\end{align*}
Next, we use \cite[Corollary~2.10]{BBNW}, which computes $\depth^D(w)$ in terms of the statistics we have defined, as well as $\neg(w) := \{i \in [n] : w(i) < 0\}$:
\begin{align*}
\$(w) & \geq \left(\frac{\tvd(w)}{2} - \neg(w) + \bl^B(w) - \bl^D(w)\right) + \neg(w)\\
&= \frac{\tvd(w)}{2} + \bl^B(w) - \bl^D(w).
\end{align*}
On the other hand, \cite[\S4.2]{BBNW} produces a factorization $t'_1 \cdots t'_{k'}$ of $w$ for which
\begin{align*}
\sum\limits_i \$(t'_i) &= \frac{\tvd(w)}{2} - \neg(w) + \bl^B(w) - \bl^D(w) + \neg(w)\\
&= \frac{\tvd(w)}{2} + \bl^B(w) - \bl^D(w).
\end{align*}
And since $\$(w) \le \sum \$(t'_i)$, we can conclude from these two inequalities that indeed
$$ \$(w) = \frac{\tvd(w)}{2} + \bl^B(w) - \bl^D(w),$$
as claimed.
\end{proof}

\section{Further remarks and open questions}
\label{sec:further remarks}

We conclude our work with commentary about our methods and a description of several possible directions for further research. Some of these possibilities involve specific conjectures, while others are more general questions or hopes for a deeper understanding.

\subsection{Different combinatorial realizations in affine types}
\label{sec:BB vs EE}

As originally observed in \cite{EE}, in the definition of $\wt{S}^C_n$, there is a choice about whether to have a mirror symmetry across the integer $n + 1$ (corresponding to the translation by $2n + 2$ in the definition) or to place the mirror between the integers $n$ and $n + 1$ (in which case the corresponding translation would be by $2n + 1$ instead).\footnote{In fact in principle one could place both mirrors between consecutive pairs of integers, so that there are no fixed points in the action of the group on $\ZZ$.  However, this clashes with the extremely natural convention to have $S^B_n$ act on $\pm[n]$ (with $0$ fixed) rather than a string of $2n$ consecutive integers like $\{-n + 1, \ldots, -1, 0, 1, \ldots, n\}$.}  Indeed, in the standard reference \cite{BB} (and in the \texttt{AffinePermutationGroup} implementation on Sage \cite{sage}), the latter convention is chosen.  This difference has no effect on the algebra of the group, but it changes the correspondence between the algebraic and combinatorial objects, and hence it changes fundamentally the answers to the questions we consider.  For example, the window notation of the $\wt{S}^C_n$-simple reflection that we denoted $s'_n$ in Section~\ref{sec:statistics} is $[1, \ldots, n - 1, n + 2]$, with $\tvd(s'_n) = 2$, but in the alternate convention it would be $[1, \ldots, n - 1, n + 1]$, with total displacement equal to $1$.  Our decision to follow the realization from \cite{EE} rather than the variation used in \cite{BB} is motivated by the fact that $\tvd(w)$ is even for every permutation, signed permutation, and affine permutation---and only in this realization is $\tvd(w)$ an even integer for every affine signed permutation $w$.

\subsection{Conjectures in affine types $\bm{B}$ and $\bm{D}$}
\label{sec:affine B and D}

There are two George groups of window size $n$ that are neither unbranched nor finite: $\affS^B_n$ and $\affS^D_n$.  They are defined relative to $\affS^C_n$ by the same sort of evenness conditions that define $S^D_n$ as a subgroup of $S^B_n$.  In particular, $\affS^B_n$ is the group of affine signed permutations $w$ such that $\# \{ i > 0 : w(i) < 0\}$ is even, and $\affS^D_n$ is the subgroup of $\affS^B_n$ consisting of those affine signed permutations $w$ such that, in addition, $\# \{i > n + 1 : w(i) < n + 1\}$ is even.

Below, we state a precise conjecture for the value $\$(w)$ when $w \in \affS^B_n$, and raise the question of computing $\$(w)$ when $w \in \affS^D_n$.  We begin by defining an analogue of blocks for affine signed permutations.

Given an affine signed permutation $w$, say that an integer $j \in [n - 1]$ is \emph{good} if $w$ restricts to a bijection on $\pm[j]$ to itself.  The number of good integers is at most $n - 1$ (achieved, for example, on the identity, but also on the permutation $[-1, -2, -3, -4] \in \wt{S}^C_4$) and can be as small as $0$ (for example, in the permutation $[4, 5]\in \wt{S}^C_2$).  Define
\[
\bl^{\wt{C}}(w) = 1 + \# \{\text{good values for } w\}.
\]
Suppose that $j$ is a good value for $w$.  Say that $j$ is (further) \emph{very good} if the restriction of $w$ to $\pm[j]$ is an even-signed permutation.  Define
\[
\bl^{\wt{B}}(w) = 1 + \# \{\text{very good values for } w\}.
\]
\begin{example}
    Consider the affine signed permutation 
    \[
    w = [1, \ol{2}, 4, 3, 6, \ol{5}, 7, \ol{8}, 10 + 24, 9, 11] \in \wt{S}^B_{11}.
    \]
    Then the good values of $w$ are $1, 2, 4, 6, 7, 8$, so that $\bl^{\wt{C}}(w) = 7$, and the very good values of $w$ are $1, 6, 7$, so that $\bl^{\wt{B}}(w) = 4$.
\end{example}
The values $\bl^{\wt{C}}(w)$ and $\bl^{\wt{B}}(w)$ are meant to be the analogue of the numbers of blocks of $w$.  The $+1$ accounts for the ``last'' block stretching out to include the value $n$---unlike the others, that block can include values in the window outside of $\pm[n]$.

The following conjecture extends Theorem~\ref{thm:finite type} to affine type $\wt{B}$.
\begin{conjecture}
\label{conj:affine B}
    If $w \in \wt{S}^B_n$, then
    \[
    \$(w) = \frac{1}{2} \tvd(w) + \bl^{\wt{C}}(w) - \bl^{\wt{B}}(w).
    \]
\end{conjecture}

If Conjecture~\ref{conj:affine B} holds, then it would follow that 
\[
\$(w) - \frac{\tvd(w)}{2} \leq n
\]
for all $w \in \wt{S}^B_n$.  By Theorems~\ref{thm:unbranched} and~\ref{thm:finite type}, the same inequality holds for $w$ in a finite or unbranched George group.  However, in $\wt{S}^D_n$, empirical evidence suggests that it is no longer true that the difference
$
\$(w) - \frac{\tvd(w)}{2}
$
is bounded as $w$ varies in the group.
\begin{conjecture}
    For all $w \in \wt{S}^D_n$, we have
    \[
    \frac{\tvd(w)}{2} \leq \$(w) \leq \tvd(w).
    \]
    Furthermore, the equality $\$(w) = \tvd(w)$ holds if and only if $w$ is of the form 
    \[
    w = [1, \ldots, i - 1, i + 2k \cdot (2n + 2), i + 1, \ldots, n]
    \]
    for some $i \in [n]$ and $k \in \ZZ$.
\end{conjecture}

More generally, we can ask the following.

\begin{question}
    Is there a formula for $\$(w)$, for $w \in \wt{S}^D_n$?
\end{question}

\subsection{An aesthetically pleasing construction}

As mentioned in Section~\ref{sec:unbranched}, the proof of our first main theorem produces a factorization of an element $w$ by successively adding factors on the right side.  This differs from the elegant approach in \cite{PT}, in which one considers some factors to have been added on the left and others to have been added on the right, thereby avoiding the need for technical lemmas akin to Lemmas~\ref{lem:unbranched can find a good transposable pair} and~\ref{lem:good transposition can be peeled off}. Is there a similarly elegant algorithm for producing / interpreting $\$$-minimizing factorizations in other types?

\subsection{Ordering the statistics}

In \cite{PT}, it is observed that various inequalities hold among the natural statistics considered, namely, that for every element $w$ of a Coxeter group,
\begin{equation}
\label{eq:chain of inequalities}
\reflen(w) \leq \frac{\reflen(w) + \ell(w)}{2} \leq \depth(w) \leq \ell(w)   
\end{equation}
where $\ell(w) = \ell_S(w)$ is the Coxeter length of $w$ (the smallest number of simple reflections whose product equals $w$) and $\reflen(w)$ is the reflection (absolute) length of $w$ (the smallest number of arbitrary reflections whose product equals $w$).  In \cite{BBNW, PT}, the elements in which various equalities in \eqref{eq:chain of inequalities} hold are classified and enumerated in finite types $A$, $B$, and $D$.

If $W$ is an \emph{unbranched} George group, so that $\tvd(s) = 1$ for every simple transposition $s$, it follows immediately from our work that \eqref{eq:chain of inequalities} can be extended as follows:
\[
\reflen(w) \leq \frac{\reflen(w) + \ell(w)}{2} \leq \depth(w) \leq \frac{1}{2}\tvd(w) = \$(w) \leq \ell(w).
\]
Can one characterize and enumerate the elements for which $\depth(w) = \$(w)$?  And, similarly, for which $\$(w) = \ell(w)$?

For the branched types, it is no longer true that either $\tvd(w)/2$ or $\$(w)$ is bounded by $\ell(w)$: indeed, this inequality is violated by the simple transposition $s = \langle(1~\ol{2})\rangle$ in these types.  Can anything interesting be said uniformly?

\subsection{Depth in affine types}

Recalling the algebraic perspective discussed in the introduction, we note that our work leaves open the question of computing the depth of affine (signed) permutations.  We mention briefly the reason we believe the answer may not be as attractive as the formulas discussed above.  In the affine symmetric group $\wt{S}_n$, the depth of a transposition $\langle (i~j)\rangle$ is given by
\[
\depth(\langle (i~j)\rangle) = \frac{1 + \ell(\langle (i~j)\rangle)}{2} = |i - j| - \left\lfloor \frac{|i - j|}{n}\right\rfloor,
\]
and similar floor terms appear in formulas in other affine types (see, e.g., \cite[(8.44)]{BB}).  Thus, even in the simplest case of the depth of a single transposition, the formula is somewhat unattractive; and it seems reasonable to expect the level of complication to grow for elements that require longer factorizations.

\section*{Acknowledgements}
The authors are grateful for the careful reading and helpful advice of the anonymous referees.

\end{document}